\documentclass[11pt]{article}
\usepackage{amsmath,amsthm,amssymb,enumerate,mathscinet}
\usepackage{fullpage}
\usepackage{color}
\newtheorem{theorem}{Theorem}[section]

\newtheorem{lemma}[theorem]{Lemma}

\newtheorem{thm}[theorem]{Theorem}

\newtheorem{defin}[theorem]{Definition}

\newtheorem{obs}[theorem]{Observation}

\def\C{\mathcal{C}}
\def\F{\mathcal{F}}
\def\HH{\mathcal{H}}
\def\P{\mathcal{P}}
\def\eps{\varepsilon}

\def\HL{\text{HL}} 

\def\COMMENT#1{}
\let\COMMENT=\footnote

\title{On the number of union-free families} 

\author{J\'ozsef Balogh\footnote{Department of Mathematical Sciences,
 University of Illinois at Urbana-Champaign, Urbana, Illinois 61801, USA {\tt
jobal@math.uiuc.edu}. Research is partially supported by Simons Fellowship, NSA Grant H98230-15-1-0002, NSF Grant
 DMS-1500121 and Arnold O. Beckman Research Award (UIUC Campus Research Board 15006).}
\ and Adam Zsolt Wagner\footnote{University of Illinois at Urbana-Champaign, Urbana, Illinois 61801, USA, {\tt
zawagne2@illinois.edu}. }}

\begin{document}
 \maketitle
\begin{abstract}
A family of sets is union-free if there are no three distinct sets in the family such that the union of two of the sets is equal to the third set. Kleitman proved that every union-free family has size at most $(1+o(1))\binom{n}{n/2}$. Later, Burosch--Demetrovics--Katona--Kleitman--Sapozhenko asked for the number $\alpha(n)$ of such families, and they proved that $2^{\binom{n}{n/2}}\leq \alpha(n) \leq 2^{2\sqrt{2}\binom{n}{n/2}(1+o(1))}$. They conjectured that the constant $2\sqrt{2}$ can be removed in the exponent of the right hand side. We prove their conjecture by formulating a new container-type theorem for rooted hypergraphs.
\end{abstract}

\section{Introduction}
Given a family $\F\subseteq \P(n)$, we say that $\F$ is \emph{union-free} if there are no three distinct sets $A,B,C\in\F$ such that $A\cup B=C$. Kleitman was a young professor at Brandeis in the early 1960s when he stumbled across a book of open mathematical problems by Ulam \cite{ulam}. The problem of determining the largest union-free family, originally raised by Erd\H{o}s, appeared in this book. Erd\H{o}s conjectured that no union-free family could have size larger than $O\left(\binom{n}{n/2}\right)$. Kleitman proved this conjecture by establishing an upper bound of $2\sqrt{2}\binom{n}{n/2}$, which he later improved to $(1+o(1))\binom{n}{n/2}$.

\begin{thm}[Kleitman \cite{kleitmanunionfree}]\label{kleitmanextremal}
If $\F\subseteq \P(n)$ is a union-free family, then $|\F|\leq (1+o(1))\binom{n}{n/2}$.
\end{thm}

Later, Burosch--Demetrovics--Katona--Kleitman--Sapozhenko raised \cite{burosch} the problem of enumerating all union-free families in $\P(n)$. Let $\alpha(n)=|\{\F\subseteq \P(n): \F \text{ is union-free}\}|$. Since the collection of all $\lfloor n/2\rfloor$-sets gives rise to a union-free family, and every subfamily of a union-free family is also union-free, we have $\alpha(n)\geq 2^{\binom{n}{\lfloor n/2\rfloor}}$. They proved the following upper bound on $\alpha(n)$:

\begin{thm}[Burosch--Demetrovics--Katona--Kleitman--Sapozhenko]
The function $\alpha(n)$ satisfies
$$2^{\binom{n}{n/2}}\leq \alpha(n) \leq 2^{2\sqrt{2}\binom{n}{n/2}(1+o(1))}.$$
\end{thm}

They conjectured that the constant $2\sqrt{2}$ in the exponent of the right hand side can be removed. The main result of this paper is that their conjecture was correct:

\begin{thm}\label{mainkleitman}
The function $\alpha(n)$ satisfies
$$\alpha(n)=2^{\binom{n}{n/2}(1+o(1))}.$$
\end{thm}

Our main tool in proving Theorem \ref{mainkleitman} is the Hypergraph Container Method, pioneered by Balogh--Morris--Samotij \cite{container1} and independently by Saxton--Thomason \cite{container2}. We create a $3$-uniform hypergraph $\HH$ with vertex set $\P(n)$, and sets $A,B,C$ forming an edge if $A\cup B=C$. Now every union-free family corresponds to an independent set in $\HH$. Hence to prove Theorem \ref{mainkleitman} we will use the Container Method to bound the number of independent sets in this hypergraph $\HH$.  The idea behind the method is that there exists a small family of vertex sets, called {\it containers}, which consists of sets spanning only few hyperedges, and each independent set is contained in one of them. 

The main difficulty in this problem compared to the main results in \cite{container1,container2}  is that here  $\HH$ does not satisfy any of the necessary co-degree conditions -- it has large subgraphs where the co-degrees are comparable to the total number of edges -- hence straightforward applications of the available container theorems are doomed to fail. To get around this difficulty we need a new version of the container theorem, that works well for \emph{rooted} hypergraphs. For this theorem to be applicable one needs to prove a nonstandard version of a supersaturation theorem. The proof of the supersaturation theorem makes use of the Expander Mixing Lemma of Alon--Chung \cite{alonchung}. 
Note that in general for the container method to work one needs some type of supersaturation, which means  that if vertex set $U$ is a somewhat larger than the independence number of the hypergraph, then $U$ contains many hyperedges.  Here we need a little bit more, we need some even distribution of these hyperedges, a similar obstacle (which was handled differently) showed up in \cite{cycle}. 

The paper is organised as follows. In Section \ref{containersection} we prove a new version of the Container Theorem. In Section \ref{supersatsection} we prove a supersaturated version of Theorem \ref{kleitmanextremal} and in Section \ref{mainresultsection} we combine it with our container theorem to prove Theorem \ref{mainkleitman}.

\section{Constructing containers in rooted hypergraphs}\label{containersection}

\begin{defin}
A $3$-uniform hypergraph $\HH$ is \textbf{rooted} if there exists a function $f:E(\HH)\rightarrow V(\HH)$ such that
\begin{itemize} 
\item for every edge $e\in E(\HH)$ we have $f(e)\in e$, and 
\item for any two vertices $u,v$ there is at most one edge $e\in E(\HH)$ with $u,v\in e$ and $f(e)\notin \{u,v\}$.
\end{itemize}
If $\HH$ is a rooted hypergraph and $f$ is specified then we call $f$ a \textbf{rooting function for} $\HH$ and for every edge $e$ we call $f(e)$ the \textbf{head} of $e$. The \textbf{head-degree} of a vertex $v$ is $\rm{hd} (v)=|\{e\in E(\HH): f(e)=v\}|$. The \textbf{head link-graph} of a vertex $v$ is the graph $\rm{HL}_v(\HH)$ with vertex set $V(\HH)$ and edge set $\{\{u_1u_2\}:\{vu_1u_2\}\in E(\HH), f\left(\{v u_1u_2\}\right)=v\}$.
\end{defin}
 
\begin{defin}
Given a $3$-uniform hypergraph $\HH$, a subset of its vertex set $A\subseteq V(\HH)$ and two positive numbers $s,t$, we say that a vertex $v\in A$ is $(A,s,t)$-\textbf{eligible} if there is a subgraph $G_v$ of its head link graph $\rm{HL}_v(\HH\cap A)$ with $\Delta(G_v)\leq s$ and $e(G_v)\geq t$.  We say that $A$ is an $(s,t)$-\textbf{core} if it does not contain an $(A,s,t)$-eligible vertex.

Given $\eps>0$ and $N>0$ we say that the hypergraph $\HH$ is $(\eps,N,s,t)$-\textbf{nice} if for every $A\subseteq V(\HH)$ with $|A|\geq (1+\eps)N$, the set $A$ contains an $(A,s,t)$-eligible vertex. 
\end{defin}

Some explanation might come in handy. Throughout the paper we will mostly work with the hypergraph $\HH$ which has vertex set $\mathcal{P}(n)$ and edge set $\{(A,B,C): A\cup B=C\}$. We prove a container theorem for general hypergraphs, but it does no harm for the reader to think of this $\HH$ throughout the proof. This hypergraph is rooted, since whenever $A\cup B=C$ we can let $f(A,B,C)=C$. The crucial observation is that given $A,C$ there may be many choices for $B$ such that $A\cup B =C$ holds (so the codegrees of the hypergraph can be very large), but given $A,B$ there is only one $C$ such that $A\cup B = C$ (hence in this direction, all codegrees are one). All our approaches using existing container lemmas broke down because $\HH$ has such large codegrees - but by breaking the symmetry and proving an oriented- (or rooted) version of the container lemma fixes the problem.

There is another difficulty that arises when one tries to prove a rooted container lemma - during the proof it is much harder to keep control over the degrees of the link graphs, when we reduce from $3$-uniform to $2$-uniform. To overcome this difficulty we need a stronger, balanced supersaturation result. A simple supersaturation result (that is not good enough for us) states that if a family has size slightly larger than the largest independent set, then it contains many edges, and hence it contains a vertex of large degree within the family. In the present paper we will show that such a family contains a vertex that not only has large degree, but in fact one can find a dense subgraph of its link graph that is nicely distributed. Using the terms defined above, our stronger supersaturation result will  show that $\HH$ is \emph{nice} (with some parameters), i.e. that if a family is slightly larger than the largest independent set then it is not a \emph{core}, so it contains an \emph{eligible} vertex (again, parameters specified later). For more details, we direct the reader to Section~\ref{supersatsection}.

\medskip

The main goal of this section is to prove the following Container Theorem.  Let $H:[0,1]\rightarrow \mathbf{R}$ be the binary entropy function defined as 
$$H(p)=-p\log p - (1-p)\log (1-p).$$

\begin{thm}
\label{containerthm}[Container theorem for rooted $3$-uniform hypergraphs]
Let $\eps,s,t,N,M>0$ be parameters satisfying  
\[ \eps\leq 1/10, \qquad  ~ {8s}\le \eps{t}, \qquad       \frac{1}{\eps^2 }\leq s \text{ ~ ~ and ~ ~ } M\geq (1+100\eps)N.
\] Let $\HH$ be a $3$-uniform rooted $M$-vertex hypergraph $\HH$ such that there exists a rooting function $f$ for $\HH$ so that $\HH$ is $(\eps,N,s,t)$-nice. Then there exists a family $\C\subseteq \P\left(V(\HH)\right)$ satisfying the following:
\begin{enumerate}
\item For every independent set $I\subseteq V(\HH)$, there exists a $C_I\in\C$ such that $I\subseteq C_I$.
\item\label{fewcontainersitem} $\log_2|\C|\leq \frac{2M}{\eps}(H(2s/t)+H(1/4\eps s))$.
\item Every $C\in\C$ satisfies $|C|\leq (1+100\eps)N$.\hfill \qed
\end{enumerate}
\end{thm}

{\bf Remark.}
1. If $\HH$ is $(\eps,N,s,t)$-nice where $s < (8 t/\eps)^{1/2}$, then $\HH$ is $(\eps,N,(8 t/\eps)^{1/2},t)$-nice
as well, and the theorem yields a smaller family $\C$ if $s$ is replaced by $(8 t/\eps)^{1/2}$.

2. Even though we will not need it in the present paper, we remark that condition (\ref{fewcontainersitem}) can be replaced by the following, stronger condition, typical for container type lemmas. Define $p$ to be the least integer such that $(1-\eps/2)^pM\leq N$. 
\begin{enumerate}
\item[\emph{2.'}] For every independent set $I\subseteq V(\HH)$, there exist fingerprints $T_I=(T_1,T_2,\ldots,T_p)$ and $S_I=(S_1,S_2,\ldots,S_p)$ with $T_i,S_i\subseteq I$ for all $i$, and fingerprint sizes $|T_i|\leq 2sM(1-\eps/2)^i/t$ and $|S_i|\leq M(1-\eps/2)^i/4\eps s$. Moreover, the container $C_I$ depends only on the pair $(T_I,S_I)$.
\end{enumerate}

The following algorithm, which is useful when the codegrees can be high but the hypergraph is rooted, will be used to produce the containers:

\vspace{.18in}

\noindent\textbf {Container algorithm}

\vspace{.18in}

\textbf{Input:} Parameters $\eps,s,t,N,\tau,z>0$ satisfying
 \[\tau\geq \frac{2s}{t},\qquad\eps\leq \frac{1}{10}, \qquad   4\eps s\geq z, \qquad \text{ and ~ } \tau+\frac{1}{z}\leq\frac{\eps}{2},\]
a $3$-uniform rooted hypergraph $\HH$ on at least $(1+100\eps)N$ vertices, a rooting function $f$ such that $\HH$ is $(\eps,N,s,t)$-nice, and an independent set $I\subseteq V(\HH)$.

\textbf{Output:}  A set $C$ such that $I\subseteq C \subseteq V(\HH)$, and two sets $T,T'\subseteq I$.

\vspace{.18in}

\textbf{Phase I}:

\begin{enumerate}
\item\label{orderinginitstep} Fix an arbitrary ordering of the vertices, and another arbitrary ordering of all graphs on vertex set $V(\HH)$. These will be used to break ties.

\item Set $A=V(\HH)$ (the set of \emph{available} vertices), $T=\emptyset$ be the \emph{fingerprint} of $I$, and let $L$ be the empty multigraph on vertex set $V(\HH)$ (the \emph{link graph} we build).

\item\label{beginning} If $|A|\leq (1-\eps)|V(\HH)|$ then set $C:=A\cup T$ and STOP. Otherwise, set $S:=\{u\in A:|N_L(u)\cap A|\geq s\}$.
\item  If $A\setminus S$ is an $(s,t)$-core then go to Phase II.
\item \label{vertexnotin} Let $v\in V(\HH[A\setminus S])$ be the largest degree vertex among $(A\setminus S,s,t)$-eligible vertices (break ties according to the ordering fixed in Step \ref{orderinginitstep}). If $v\notin I$ then replace $A$ by $A\setminus \{v\}$, replace $L$ by $L\setminus \{v\}$ and return to Step \ref{beginning}.
\item\label{vertexin} We have $v\in I$ and $v$ is $(A\setminus S,s,t)$-eligible. Let $G_v$ be a subgraph of its head link graph $\HL_v(\HH)\cap A\setminus S$ with $\Delta(G_v)\leq s$ and $e(G_v)\geq t$ (break ties according to the ordering fixed in Step \ref{orderinginitstep}). 
\begin{enumerate}
\item Set $T:=T\cup\{v\}$.
\item Set $A:=A\setminus \{v\}$.
\item Let $L:=\left(L\cup G_v\right)\setminus \{v\}$ (considered as a multigraph).
\item Return to Step \ref{beginning}.
\end{enumerate}
\end{enumerate}

\vspace{.18in}

\textbf{Phase II}:

\begin{enumerate}
\item Initiate $T'=\emptyset$, the second \emph{fingerprint}. 
\item\label{beginning2} Let $v$ be the largest degree vertex in $L$ (break ties according to the ordering fixed in Step I.\ref{orderinginitstep}). If $d_L(v)<z$ then set $C:=A\cup T\cup T'$ and STOP.
\item If $v\notin I$ then replace $A$ by $A\setminus \{v\}$ and replace $L$ by $L\setminus \{v\}$ and go to Step \ref{beginning2}.
\item We have $v\in I$ and $d_L(v)\geq z$. Set $T':=T'\cup\{v\}$, replace $A$ by $A\setminus\left(N_L(v)\cup\{v\}\right)$ and  replace $L$ by $L\setminus\left(N_L(v)\cup\{v\}\right)$. Go to Step \ref{beginning2}.
\end{enumerate}

\textbf{End of algorithm.}

\vspace{.18in}

\begin{obs}\label{containermethodworks}
The containers only depend on the fingerprints $T,T'$. 
\end{obs}
\begin{proof}
Person A runs the algorithm with input $I$ and gets output $C,T,T'$. He then tells Person B the values of $T$, $T'$ and all other input parameters (including the orderings specified in Step \ref{orderinginitstep}), but not $I$. We claim that B can find the value of $C$. Indeed, all he has to do is to try to follow the algorithm exactly as A did. In Phase I, the only critical points are in Steps \ref{vertexnotin} and \ref{vertexin} where B seems to need knowledge of $I$ to make the same decisions as A did. But actually all B needs to know is whether the vertex $v$ is in $I$ or not. But for this $v$ we know that $v\in T$ iff $v\in I$, hence B can run Phase I the same way as A did (and hence at every point in Phase I B will know the values of $L,S,A$, etc.).

The same argument applies to Phase II. The largest degree vertex $v$ that we consider in the algorithm is in $I$ precisely if it gets put into the fingerprint $T'$, hence B can recover $C$.
\end{proof}

\begin{obs}\label{smallmultiplicityL}
At every point in the algorithm, $L$ is a simple graph.
\end{obs}
\begin{proof}
Recall that the input graph $\HH$ was rooted. Hence every pair of vertices forms an edge in the head link graph of at most one other vertex.
\end{proof}

\begin{obs}\label{Lmaxdegree}
At every point in the algorithm, $\Delta(L)\leq 2s$.
\end{obs}
\begin{proof}
Every time we change $L$, we add to it a graph of maximum degree at most $s$. But as soon as some vertex gets degree at least $ s$ we put it in $S$ and do not touch it until its degree goes below $s$ again. Hence in $L$, the maximum possible degree is at most $s+s=2s$.
\end{proof}

\begin{obs}\label{smallfingerprints}[Small fingerprints.]
After the algorithm stops we have $|T|\leq \tau|V(\HH)|$ and $|T'|\leq |V(\HH)|/z$.
\end{obs}
\begin{proof}
Suppose we have $|T|> \tau |V(\HH)|$ at some point in the algorithm. Stop the algorithm when this happens (in step $6(d)$) and count the edges in $L$. Every time we increase $T$ we add at least $t$ edges to $L$. The only times we delete edges from $L$ is when we remove some vertices from $A\setminus S$. But these vertices had by definition at most $s$ neighbours in $L$. Hence in total we remove at most $|V(\HH)|s$ edges from $L$, and so $e(L)> \tau|V(\HH)|t-|V(\HH)|s$. By Observation \ref{Lmaxdegree}, we have $\Delta(L)\leq 2s$, so $s|V(L)|\geq e(L)>  \tau|V(\HH)|t-|V(\HH)|s$ whence it follows that $2s/t> \tau$ (since $L\subseteq \HH$). As in the input we took  $\tau \geq 2s/t$ we conclude that $|T|< \tau |V(\HH)|$ at every point in the algorithm.

If the algorithm stopped in Phase I then $T'=\emptyset$ and the claim follows. Otherwise, every time we put a vertex into $T'$ we removed at least $z$ vertices from $A$, by Observation \ref{smallmultiplicityL}. Hence $|T'|\leq |V(\HH)|/z$ and the claim follows.
\end{proof}

\begin{obs}\label{smallhalfcontainer}
After the algorithm stops we have $|A|\leq (1-\eps)|V(\HH)|$.
\end{obs}
\begin{proof}
If the algorithm terminated in Phase I then the claim follows. Now assume the algorithm entered Phase II. Denote $A_1,S,L$ the sets $A,S$ and graph $L$ right before entering Phase II of the algorithm, and let $A_2$ be the set $A$ at the end of the algorithm (noting that $A_2\subseteq A_1$). Since we did not terminate after Phase I, we know that $|A_1|\geq (1-\eps)|V(\HH)|$. It was specified in the input that $|V(\HH)|\geq (1+100\eps)N$ and $\eps\leq1/10$, hence we conclude that 
\begin{equation}\label{A1isbig}
   |A_1|\geq (1+89\eps)N.
 \end{equation}
Since $\HH$ is $(\eps, N, s,t)$-nice, we have 
\begin{equation}\label{AminusSbig}
|A_1\setminus S|\leq (1+\eps)N,
\end{equation} 
otherwise $A_1\setminus S$ would not be an $(s,t)$-core. Since $\eps\leq 1/10$ we also have
\begin{equation}\label{SperAbig}
\frac{|S|}{|A_1|}\overset{(\ref{AminusSbig})}{\geq}1-\frac{(1+\eps)N}{|A_1|}\overset{(\ref{A1isbig})}{\geq} 1-\frac{1+\eps}{1+89\eps}> 8\eps. 
\end{equation}
Since $|A_1|\leq |V(\HH)|$, if it is the case that $|A_2|\leq (1-\eps)|A_1|$ then we are done. Hence in what follows, we assume for contradiction that $|A_2|> (1-\eps)|A_1|$, implying
\begin{equation}\label{AminusAbig}
|A_1\setminus A_2|< \eps |A_1|.
\end{equation}
As in $L[A_2]$ every vertex has degree at most $z$, we get 
\begin{equation}\label{eLsmall1}
e(L[A_2])\leq \frac{|A_2|z}{2}\leq\frac{|A_1|z}{2}.
\end{equation}
 Recall that in $L$, every vertex has degree at most $2s$. Hence, counting those edges in $L$ which have at least one endpoint in $A_1\setminus A_2$ we get 
\begin{equation}\label{eLsmall2}
e(L[A_1])-e(L[A_2])\leq |A_1\setminus A_2|2s\overset{(\ref{AminusAbig})}{<} \eps|A_1|2s.
\end{equation} 
Note also that in $L$, every vertex in $S$ has degree at least $s$. As $e(L)=e(L[A_1])$, we have 
\begin{equation}\label{eLbig}
e(L)\geq \frac{|S|s}{2}.
\end{equation}
Putting the relations (\ref{eLsmall1}), (\ref{eLsmall2}) and (\ref{eLbig}) together we get
\begin{equation}\label{SperAsmall}
\frac{|S|s}{2}\le e(L) =e(L(A_1)) \le e(L(A_2))+    \eps|A_1|2s        < \frac{|A_1|z}{2}+ \eps|A_1|2s\overset{(\ref{SperAbig})}{\leq}  \frac{|A_1|z}{2}+\frac{|S|s}{4}.
\end{equation}
Hence
\begin{equation*}
8\eps\overset{(\ref{SperAbig})}{<}\frac{|S|}{|A_1|}\overset{(\ref{SperAsmall})}{<}\frac{2z}{s},
\end{equation*} 
which contradicts the restriction $4\eps s\geq z$ on the input parameters. This completes the proof.
\end{proof}

\begin{obs}\label{smallcontainer}[Small containers.]
After the algorithm stops we have $|C|\leq \left(1-\frac{\eps}{2}\right)|V(\HH)|$.
\end{obs}
\begin{proof}
If the algorithm stopped after Phase I then $|C|\leq |A|+|T|$, and if the algorithm stopped after Phase II we get $|C|\leq |A|+|T|+|T'|$. In both cases, we have $|C|\leq (1-\eps)|V(\HH)|+\tau|V(\HH)|+|V(\HH)|/z$ by Observations \ref{smallfingerprints} and \ref{smallhalfcontainer}. Since $\tau+1/z\leq\eps/2$ we have $|C|\leq \left(1-\frac{\eps}{2}\right)|V(\HH)|$ as required. 
\end{proof}

Putting all these observations together, we get the following container lemma. We use the notation $\binom{M}{\leq m}=\sum_{i=0}^m \binom{M}{i}$.

\begin{lemma}\label{maincontainer}[Container lemma for rooted $3$-uniform hypergraphs]
Let $\eps,s,t,N,M>0$ be parameters satisfying 
 \[\eps\leq 1/10, \qquad  
  {8s}\le \eps{t}, \qquad       \frac{1}{\eps^2 }\leq s   
     \qquad \text{ and ~ } M\geq (1+100\eps)N.\] Let $\HH$ be a $3$-uniform rooted $M$-vertex hypergraph $\HH$ such that there exists a rooting function $f$ for $\HH$ so that $\HH$ is $(\eps,N,s,t)$-nice. Then there exists a family $\C\subseteq \P\left(V(\HH)\right)$ satisfying the following:
\begin{enumerate}
\item For every independent set $I\subseteq V(\HH)$, there exists a $C\in\C$ such that $I\subseteq C$.
\item $|\C|\leq \binom{M}{\leq2s M/t}\binom{M}{\leq M/4\eps s}$.
\item Every $C\in\C$ satisfies $|C|\leq (1-\eps/2)n$.
\end{enumerate}
\end{lemma}
\begin{proof}
Setting $\tau=2s/t$ and $z=4\eps s$, the claim follows from Observations \ref{containermethodworks} - \ref{smallcontainer}.
\end{proof}

We will obtain our main container theorem by iterating Lemma \ref{maincontainer}.
Recall  the following standard bound on the sum of binomial coefficients that for  $\zeta\leq 1/2$ and $M$
\begin{equation}\label{entropybinombound}
\binom{M}{\leq \zeta M}\leq 2^{H(\zeta)M}.
\end{equation}

\begin{proof}[Proof of Theorem \ref{containerthm}]
The key observation that makes this proof work is that if $\HH$ is $(\eps,N,s,t)$-nice then for any $S\subseteq V(\HH)$ we have that $\HH[S]$ is $(\eps,N,s,t)$-nice. Hence we can iterate Lemma~\ref{maincontainer} to obtain a family $\C$, all of whose members have sizes less than $(1+100\eps)N$. 

Set \[\tau=2s/t, \qquad \beta=1/(4\eps s) \qquad \text{ and ~} \gamma = 1- \eps/2.\] The size of $\C$ satisfies, for some $N'$ with $N<N'<N(1+100\eps)$
\begin{equation*}
\begin{split}
|\C|&\leq \binom{M}{\leq \tau M}\binom{M}{\leq \beta M}\cdot  \binom{M\gamma}{\leq \tau M\gamma}\binom{M\gamma}{\leq \beta M\gamma}\cdot \binom{M\gamma^2}{\leq \tau M\gamma^2}\binom{M\gamma^2}{\leq \beta M\gamma^2}\cdot\ldots\cdot\binom{N'}{\leq \tau N'}\binom{N'}{\leq\beta N'}\\
&\overset{(\ref{entropybinombound})}{\leq} 2^{(H(\tau)+H(\beta))\left(M+M\gamma+M\gamma^2+\ldots+N'\right)}\leq 2^{2M(H(\tau)+H(\beta))/\eps},
\end{split}
\end{equation*}
and the result follows.
\end{proof}

\section{A supersaturated version of Theorem \ref{kleitmanextremal}}\label{supersatsection}

In this section we will consider a family $\F$ of size slightly larger than the maximal size of a union-free family, say $|\F|=(1+\eps)\binom{n}{n/2}$. Then by Theorem \ref{kleitmanextremal} we know that $\F$ contains a triple $A,B,C$ with $A\cup B=C$. With more work one can prove that $\F$ contains at least $\eps'n^2\binom{n}{n/2}$ such triples, where $\eps'$ is a constant depending on $\eps$. Note also that the factor $n^2$ cannot be improved to $n^{2+\alpha}$ for some constant $\alpha>0$, as if $\F$ is contained in the middle two layers of $\P(n)$ then every element in $\F$ has at most $n$ subsets in $\F$, and hence for fixed $C$ the equation $A\cup B=C$ has at most $n^2$ solutions. 

Unfortunately this supersaturation is not quite strong enough for us - we not only want to find many triples in $\F$, but we want to find a large subset of such triples that is nicely distributed. Let $\HH$ be the $3$-uniform hypergraph on vertex set $\P(n)$, three sets $A,B,C$ forming an edge with head $A$ if $C=A\cup B$. We want to prove that $\HH[\F]$ contains at least one $(\F,n,\eps'n^2)$-eligible vertex. (Note that it is then an immediate corollary that $\HH[\F]$ contains at least $\eps''n^2\binom{n}{n/2}$ edges for some $\eps''>0$.)

\begin{thm}\label{supersat}
Let $0<\eps<1/200$ be a small constant and $n$ sufficiently large.  If $|\mathcal{F}|\geq\binom{n}{n/2}(1+\eps)$ then $\HH[\mathcal{F}]$ contains at least one $(\F,n,\frac{\eps^2}{10^{40}} n^2)$-eligible vertex.
\end{thm}

 The first ingredient in the proof is the Expander Mixing Lemma, due to Alon and Chung \cite{alonchung}:

\begin{thm}[Expander Mixing Lemma]

Let $G$ be a $D$-regular graph on $N$ vertices, and let $\lambda$ be its minimum eigenvalue. Then for all $S\subseteq V(G)$,
$$e(G[S])\geq \frac{D}{2N}|S|^2+\frac{\lambda}{2N}|S|(N-|S|).$$
\end{thm} 

Denote $\rm{KG}(m,k)$ the Kneser graph with vertex set $\binom{[m]}{k}$, two $k$-sets being connected by an edge if the sets are disjoint. Then $\rm{KG}(m,k)$ is $D$-regular with $D=\binom{m-k}{k}$ and its minimum eigenvalue $\lambda=-\frac{k}{m-k}D$ (see \cite{lovasz}). Let $N=\binom{m}{k}=|V(KG(m,k))|$. The following is a corollary of the Expander Mixing Lemma. 

\begin{lemma}\label{michelleeml}
Given $\beta>0$, any set $S$ of at least $(1+\beta)\binom{m-1}{k-1}$ vertices in $\rm{KG}(m,k)$ induces at least $\left(1-\frac{1}{1+\beta}\right)\frac{Dm}{N(m-k)}\binom{|S|}{2}$ edges.\hfill \qed
\end{lemma}

We will need the following easy lemma to take care of families which are densely packed on the middle layers of $\P(n)$. Families which are more spread out will be much harder to handle. Recall that $\HH$ is the $3$-uniform hypergraph with vertex set $\P(n)$, and sets $A,B,C$ forming an edge if $A\cup B=C$. 
If for two sets $A,B$ we have $A\subseteq B$ or $B\subseteq A$ then we call $(A,B)$ a \emph{comparable pair}. Let $\mathcal{F}$ be a family in $\P (n)$, and for $i\in [n]$ let $B_i$ denote the number of comparable pairs $A,B\in\mathcal{F}$ with $|B\backslash A|=i$. For any $A\in\F$, we  write $B_i(A)=\{B\in \F: B\subseteq A, |A\backslash B|=i\}$. 

\begin{lemma}\label{smalldistanceeligible}
Let  $0<\delta<1/10$, ~$n>n_0(\delta)$ sufficiently large, $\F\subseteq \P(n)$, ~  $k\in\{1,2,\ldots,10\}$ and $A\in\F$ with $n-\sqrt{n\log n}<2|A|<n+\sqrt{n\log n}$.  Suppose $|B_k(A)|\geq \delta n^k$. Then $A$ is $\left( \F, n, \delta^2n^2 \right)$-eligible.
\end{lemma}
\begin{proof}
If $k=1$ then any two sets $C_1,C_2\in B_1(A)$ satisfy $C_1\cup C_2=A$. As $|B_1(A)|\leq |A|<n$, the claim follows. So now assume $k\geq 2$.

Let $G$ be the graph on vertex set $V(G)=B_k(A)$, two sets $B_1,B_2$ being connected by an edge in $G$ if $B_1\cup B_2=A$. Note that $B_1$ and $B_2$ are connected by an edge in $G$ iff $(A\setminus B_1)\cap(A\setminus B_2)=\emptyset$. We want to estimate the number of edges in $G$ using Lemma \ref{michelleeml}, hence we define the graph $G'$ on vertex set $V(G')=\{A\setminus B : B\in \F, ~ B\subseteq A, ~ |A\setminus B|=k\}$, with edges connecting two sets precisely if they are disjoint. Note that by the above remark, the graphs $G$ and $G'$ are isomorphic.

Let $\delta'$ be defined by $\delta n^k=\delta'\binom{|A|}{k}$, hence $\delta'\geq \delta$. Define $\beta$ by
$$1+\beta=\frac{|V(G')|}{\binom{|A|-1}{k-1}}\geq \delta'\frac{|A|}{k}\geq \delta\frac{|A|}{k}.$$
Now we can apply Lemma \ref{michelleeml} to conclude that the number of edges in $G'$ is at least
$$e(G')\geq \left(1-\frac{k}{\delta |A|}\right)\frac{\binom{|A|-k}{k}|A|}{\binom{|A|}{k}(|A|-k)}\binom{|V(G')|}{2}.$$
Choosing $n\gg \delta^{-1}$ we get that $\frac{k}{\delta|A|}<1/2$, ~ $\binom{|A|-k}{k}/\binom{|A|}{k}\geq 1/2$ and $|A|/(|A|-k)\geq 1/2$. Hence if $n$ is sufficiently large we get
$$e(G')\geq \frac{1}{8}\binom{|V(G')|}{2}.$$
Now let $\mathbf{X}$ be a random $n$-vertex subgraph of $G'$. The expected number of edges in $G'[\mathbf{X}]$ is at least $\binom{n}{2}/8$, hence there exists a subgraph $G''$ of $G'$ with $|V(G'')|=n$ and $e(G'')\geq \binom{n}{2}/8$, and the claim follows.
\end{proof}

At the very end of the supersaturation proof, we will make use of the following embedding lemma.

\begin{lemma}\label{stupidembedding}
Let $m$ be a positive integer and let $G$ be a graph with $|V(G)|\geq m$. Let $S\subseteq G$ be the set of the $m$ largest degree vertices in $G$. Suppose $e(G\setminus S)\geq m^2$. Then there exists a subgraph $H\subseteq G$ such that $\Delta(H)\leq m$ and $e(H)\geq m^2/2$.
\end{lemma}
\begin{proof}
If $\Delta(G\setminus S)\leq m$ then $H=G\setminus S$ will satisfy the claim. Now assume $\Delta(G\setminus S)\geq m$. Then for each $v\in S$ we have $d_G(v)\geq m$, as $S$ was the collection of the largest degree vertices. We will build $H$  in $m$ steps. Initially, let $H$ be the empty graph on vertex set $V(H)=V(G)$ and let $S=\{s_1,s_2,\ldots,s_m\}$. In step $i$, let $N_i=N(s_i)\setminus\{s_1,s_2,\ldots,s_{i-1}\}$ and let $N'_i\subset N_i$ be any subset with $|N_i'|=m-i+1$. For each $v\in N'_i$ add the edge $(s_iv)$ to $H$.

The algorithm finishes in $m$ steps, and we have added a total of $m(m+1)/2$ edges to $H$. Each vertex not in $S$ receives at most one edge each step, hence their degree never goes above $m$. A vertex $s_i\in S$ receives at most one edge in steps $1,2,\ldots,i-1$, it receives $m-i+1$ edges in step $i$, and none after. Hence the graph $H$ constructed this way satisfies all conditions.
\end{proof}

Now we are ready to prove a supersaturated version of Theorem \ref{kleitmanextremal}.

\begin{proof}[Proof of Theorem \ref{supersat}]

Let $\Delta=\frac{n}{2}-\sqrt{n\log n}$, ~ $\Delta_1=\frac{n}{2}-\frac{1}{2}\sqrt{n\log n}$ and $\Delta_2=\frac{n}{2}+\frac{1}{2}\sqrt{n\log n}$. Note that 
$$|\{A\in\P(n): |A|\geq \Delta_2 \text{ or } |A|\leq \Delta_1\}|=o\left(\binom{n}{n/2}\right),$$
hence setting $\F_1=\{A\in\F:\Delta_1\leq|A|\leq\Delta_2\}$, for n sufficiently large we have
$$|\F_1|\geq (1+\eps/2)\binom{n}{n/2}.$$
Replace $\F$ by a subset of $\F_1$ of size $(1+\eps/2)\binom{n}{n/2}$, and so from now on we will work with a family of size  $(1+\eps/2)\binom{n}{n/2}$ with all members having size between $\Delta_1$ and $\Delta_2$, that we still call $\F$.

Given a permutation $\Pi\in S_n$ and a set $A\in \F$, we say the pair $(\Pi,A)$ is \emph{good} if
\begin{enumerate}
\item[(i)] The elements of $A$ form the first $|A|$ elements of $\Pi$;
\item[(ii)] For every $B\subset A$, if $B$ is an initial segment of $\Pi$ then $B\notin \mathcal{F}$.
\end{enumerate}

We say a pair $(\Pi,A)$ is \emph{bad} if condition $(i)$ above holds, but $(ii)$ does not. We say a bad pair $(\Pi,A)$ is \emph{horrible} if there is a $B\in\mathcal{F}$ with $B\subset A$, $B$ is an initial segment of $\Pi$ and $|A\backslash B|\geq 11$. 

Now fix a set $A\in \F$. We will say that $\Pi$ is bad/horrible if $(\Pi,A)$ is a bad/horrible pair. Let $$\mathcal{H}_A=\{C\in\P(n) : |C|=\Delta,\  C \text{ is the set of the first } \Delta \text{ elements in a horrible permutation } \Pi\}.$$
 So although $C$ does not lie in $\mathcal{F}$, we have that $C\subseteq B\subsetneq A$ for some $B\in\mathcal{F}$, with $|A\backslash B|\geq 11$. Define $\alpha_A$ by the equation
\begin{equation}\label{hadefinition}
|\mathcal{H}_A|=\binom{|A|-1}{\Delta}(1+\alpha_A).
\end{equation}
Note that $\alpha_A\geq -1$ for all $A\in \F$. 

Let $S_A$ be the number of bad permutations for the set $A$. Note that if for any $k\in\{1,2,\ldots,10\}$ we have $|B_k(A)|\geq \frac{\eps}{10^{20}} n^k$ then the claim follows by Lemma \ref{smalldistanceeligible}, hence we may assume this is not the case.  Then we have
\begin{equation}\label{sainequality}
\begin{split}
\frac{S_A}{(n-|A|)!} & \leq \frac{\eps n}{10^{20}} \left((|A|-1)!+n(|A|-2)!\cdot 2!+\ldots+n^9(|A|-10)!\cdot 10!\right)+|\mathcal{H}_A|\Delta ! (|A|-\Delta)! \\ & \leq \frac{\eps}{100}|A|!+|\mathcal{H}_A|\Delta ! (|A|-\Delta)!. 
\end{split}
\end{equation}
Since every permutation is in at most one good pair, we get that
$$\sum_{A\in\mathcal{F}} \left(|A|!(n-|A|)!-S_A\right)\leq n!.$$
Dividing by $n!$ and using (\ref{hadefinition}) and (\ref{sainequality}) we have
$$\sum_{A\in\mathcal{F}}\frac{1-\eps/100-(1-\frac{\Delta}{|A|})(1+\alpha_A)}{\binom{n}{|A|}}\leq 1.$$
Define $\beta_A=\max\{\alpha_A,0\}$. Rearrange the enumerator and use that $\binom{n}{|A|}\leq \binom{n}{n/2}$ to get
\begin{equation}\label{betaalphainequality}\sum_{A\in\mathcal{F}}\left(\frac{\Delta}{|A|}-\frac{\eps}{100}\right) -\beta_A\left(1-\frac{\Delta}{|A|}\right)\leq \binom{n}{n/2}.
\end{equation}
Now $\frac{\Delta}{|A|}\geq \frac{\frac{n}{2}-\sqrt{n\log n}}{\frac{n}{2}+\frac12\sqrt{n}\log n}\geq 1-\frac{3\sqrt{\log n}}{\sqrt{n}}$, hence $\frac{\Delta}{|A|}-\frac{\eps}{100}\geq 1-\frac{\eps}{99}$. So (\ref{betaalphainequality}) becomes
$$\left(1-\frac{\eps}{99}\right)|\mathcal{F}|-\binom{n}{n/2}\leq \sum_{A\in\mathcal{F}}\beta_A\frac{3\sqrt{\log n}}{\sqrt{n}}.$$
Since $|\mathcal{F}|\geq \left(1+\eps/2\right)\binom{n}{n/2}$ and $(1-\eps/99)(1+\eps/2)\geq(1+\eps/4)$,
 we conclude that
$$\frac{\eps}{12}\frac{\sqrt{n}}{\sqrt{\log n}}\binom{n}{n/2}\leq \sum_{A\in\mathcal{F}} \beta_A\leq \sum_{A\in\mathcal{F}} (\alpha_A+1).$$
Therefore the average value of the $\alpha_A$-s, denoted by $\bar{\alpha}_A$, satisfies
$$\bar{\alpha}_A=\frac{1}{|\mathcal{F}|}\sum_{A\in\mathcal{F}}\alpha_A\geq \frac{\eps}{20}\frac{\sqrt{n}}{\sqrt{\log n}}.$$
Hence there exists at least one $A^*\in \mathcal{F}$ such that $\alpha_{A^*}\geq \frac{\eps}{20}\frac{\sqrt{n}}{\sqrt{\log n}}$. 
Now we claim that this $A^*$ is $(\F,n,n^2/2)$-eligible. First note that 
$$|\HH_{A^*}|=\binom{|A^*|-1}{\Delta}(1+\alpha_{A^*})\geq \binom{|A^*|}{\Delta}\frac{\eps\left(|A^*|-\Delta\right)}{20|A^*|}\frac{\sqrt{n}}{\sqrt{\log n}}.$$
Since every set in $\F$ have size between $\Delta_1$ and $\Delta_2$ we get
$$|\HH_{A^*}|\geq \binom{|A^*|}{\Delta}\frac{\eps\frac12\sqrt{n\log n}}{20n}\frac{\sqrt{n}}{\sqrt{\log n}}\geq \frac{\eps}{50}\binom{|A^*|}{\Delta}.$$
Consider the graph $G$ with vertex set $V(G)=\{B\in\F:B\subsetneq A^*, |A^*\setminus B|\geq 11\}$ and an edge connecting $B_1,B_2$ if $B_1\cup B_2=A^*$.
Let $S$ be the set of the $n$ largest degree vertices in $G$. Then
$$|S_\Delta|:=|\{C\in\P(n)~:~|C|=\Delta,~\exists B\in  S: C\subset B \}|\leq n\binom{|A^*|-11}{\Delta}\leq \frac{1}{n}\binom{|A^*|}{\Delta}.$$
Let $\HH_{A^*}'=\HH_{A^*}\setminus S_\Delta$,~ $G'=G\setminus S$ and note that 
$$|V(\HH_{A^*}')|= |V(\HH_{A^*})|-|S_\Delta|\geq\frac{\eps}{100}\binom{|A^*|}{\Delta}.$$
Let us now count the number $P$ of pairs $B_1,B_2\in\HH_{A^*}'$ satisfying $B_1\cup B_2=A^*$. Note that $B_1\cup B_2=A^*$ iff $(A^*\setminus B_1)\cap (A^*\setminus B_2)=\emptyset$. Define $\beta$ by
$$1+\beta=\frac{|V(\HH_{A^*})|}{\binom{|A^*|-1}{|A^*|-\Delta-1}}>1,$$
hence we get by Lemma \ref{michelleeml} with $N=\binom{|A^*|}{|A^*|-\Delta},$   $D=\binom{\Delta}{|A^*|-\Delta}$,  $m=|A^*|$ and $k=|A^*|-\Delta$ that
\begin{equation*}
P\geq \left(1-\frac{\binom{|A^*|-1}{|A^*|-\Delta-1}}{\frac{\eps}{100}\binom{|A^*|}{\Delta}}\right)
\frac{\binom{\Delta}{|A^*|-\Delta}|A^*|}{\binom{|A^*|}{|A^*|-\Delta}\Delta}\binom{|V(\HH_{A^*})|}{2}.
\end{equation*}
Note that as $n-\sqrt{n\log n}\leq 2|A^*|\leq n+\sqrt{n\log n}$ and $\Delta=n/2-\sqrt{n\log n}$ we get
\begin{equation*}
\begin{split}
\frac{\binom{\Delta}{|A^*|-\Delta}}{\binom{|A^*|}{|A^*|-\Delta}}&\geq \left(\frac{\Delta-(|A^*|-\Delta)}{|A^*|-(|A^*|-\Delta)}\right)^{|A^*|-\Delta}=\left(1-\frac{|A^*|-\Delta}{\Delta}\right)^{\frac{\Delta}{|A^*|-\Delta}\frac{(|A^*|-\Delta)^2}{\Delta}}\\
&\geq e^{-1.01\frac{(3\sqrt{n\log n}/2)^2}{n/2.01}}\geq e^{-\frac{19}{4}\log n}=\frac{1}{n^{19/4}}.
\end{split}
\end{equation*}
Hence we have
$$P\geq \frac{1}{n^5}\binom{|V(\HH_{A^*})|}{2}.$$
Since every vertex in $G'$ corresponds to at most $\binom{|A^*|-11}{\Delta}$ vertices in $\HH_{A^*}$, the number of edges in $G'$ is at least
$$e(G')\geq \frac{1}{10n^5}\left(\frac{|V(\HH_{A^*})|}{\binom{|A^*|-11}{\Delta}}\right)^2\geq \frac{\eps^2}{10^6n^5}\left(\frac{|A^*|}{|A^*|-\Delta}\right)^{22}\geq\frac{\eps^2}{10^{20}n^5}\left(\frac{n}{\sqrt{n\log n}}\right)^{22}\geq\frac{\eps^2 n^6}{10^{20}\log^{11}n}\geq n^2.$$
Hence in $G$ the $n$ largest degree vertices are in $S\subset V(G)$ and after removing $S$ from $G$ we still have $e(G\setminus S)=e(G')\geq n^2$ edges. By Lemma \ref{stupidembedding} there is a subgraph $G^*\subseteq G$ with $\Delta(G^*)\leq n$ and $e(G^*)\geq n^2/2$. So $A^*$ is $(\F,n,n^2/2)$-eligible and the proof is completed.
\end{proof}

\section{Proof of the main result}\label{mainresultsection}

\begin{proof}[Proof of Theorem \ref{mainkleitman}]
Define the $3$-uniform hypergraph $\HH$ on vertex set $\P(n)$ and edge set $E(\HH)=\{(A,B,C):A\cup B=C\}$. For an edge $(A,B,C)$ with $A\cup B=C$ set $f(A,B,C)=C$. Hence $\HH$ is rooted under $f$. Fix a constant $\eps$ with $0<\eps<1/200$ and let $n$ be sufficiently large. By Theorem \ref{supersat}, $\HH$ is $\left(\eps,\binom{n}{n/2},n,\frac{\eps^2}{10^{40}} n^2\right)$-nice. Apply Theorem \ref{maincontainer} with parameters  $s=n$, $t=\frac{\eps^2}{10^{40}} n^2$, $N=\binom{n}{n/2}$  to obtain a family $\C$ of containers. Each container $C\in\C$ satisfies
$$|C|\leq (1+100\eps)\binom{n}{n/2},$$ and the size of the family of containers satisfies
$$\log_2|\C|\leq \frac{2\cdot2^n}{\eps}\left(H\left(2\cdot 10^{40}/\eps^2n\right)+H\left(1/4\eps n\right)\right).$$
Since $H(x)\leq 2x\log(x^{-1})$ for $x<1/2$, we get
$$\log_2|\C|\leq 10^{42}\frac{2^n}{\eps}\left(\frac{\log\left(\eps^2n\right)}{\eps^2n}+\frac{\log(4\eps n)}{4\eps n}\right)\leq 10^{44}\eps^{-3}\frac{2^n}{n}\log n\leq\eps\binom{n}{n/2}.$$
The number of independent sets in $\HH$, and hence the number $\alpha(n)$ of union-free families is bounded by
$$\alpha(n)\leq 2^{\eps\binom{n}{n/2}}2^{(1+100\eps)\binom{n}{n/2}}\leq  2^{(1+101\eps)\binom{n}{n/2}}.$$
This completes the proof of Theorem \ref{mainkleitman}.
\end{proof}

\section{Concluding remarks}

Instead of the definition of a \emph{rooted} hypergraph we could have defined more generally a $r$-\textbf{rooted} hypergraph. A $3$-uniform hypergraph $\HH$ is $r$\textbf{-rooted} if there exists a function $f:E(\HH)\rightarrow V(\HH)$ such that
\begin{itemize} 
\item for every edge $e\in E(\HH)$ we have $f(e)\in e$, and 
\item for any pair of  vertices $u,v$ there exist at most $r$ edges $e\in E(\HH)$ with $u,v\in e$ and $f(e)\notin \{u,v\}$.
\end{itemize}
Similarly as before, if $\HH$ is an $r$-rooted hypergraph and $f$ is specified then we call $f$ a \textbf{rooting function for} $\HH$. 
Given this definition, essentially the same proof gives the following container theorem:

\begin{thm}\label{rrootedcontainerthm}
\label{containerthm}[Container theorem for $r$-rooted $3$-uniform hypergraphs]
Let $\eps,s,t,r,N,M>0$ be parameters satisfying  
\[ \eps\leq 1/10, \qquad  ~ \frac{4s}{t}+\frac{r}{2\eps s}\leq\eps \text{ ~ ~ and ~ ~ } M\geq (1+100\eps)N.
\] Let $\HH$ be a $3$-uniform $r$-rooted $M$-vertex hypergraph $\HH$ such that there exists a rooting function $f$ for $\HH$ so that $\HH$ is $(\eps,N,s,t)$-nice. Then there exists a family $\C\subseteq \P\left(V(\HH)\right)$ satisfying the following:
\begin{enumerate}
\item For every independent set $I\subseteq V(\HH)$, there exists a $C_I\in\C$ such that $I\subseteq C_I$.
\item\label{fewcontainersitem} $\log_2|\C|\leq \frac{2M}{\eps}(H(2s/t)+H(r/4\eps s))$.
\item Every $C\in\C$ satisfies $|C|\leq (1+100\eps)N$.\hfill \qed
\end{enumerate}
\end{thm}

We note that Theorem \ref{rrootedcontainerthm} can be generalised to $k$-uniform hypergraphs, but due to lack of applications we chose not to do so here. It is a natural question to ask how Theorem \ref{rrootedcontainerthm} compares to the vast number of container lemmas in the literature. The primary difference is that our lemma works very well if the codegrees of a hypergraph are high, but the edges can be oriented in such a way that in one direction all codegrees are small, as is the hypergraph $\HH$ considered throughout this paper (indeed the reason why we proved Theorem \ref{rrootedcontainerthm} in the first place was that we were not able to prove Theorem \ref{mainkleitman} using any of the already existing container lemmas). Other than this difference, the proof of Theorem \ref{rrootedcontainerthm} resembles the main theorems of \cite{container1,container2}, but we put more effort into calculating the actual dependence of the various constants.

\section*{Acknowledgments}
We are very grateful to Andrew Treglown, who participated in many fruitful discussions at the beginning of the project, and Gyula Katona for many helpful discussions.

\end{document}